\documentclass{amsart}
\usepackage{amsmath}
\usepackage{amsfonts}

\newtheorem{theorem}{Theorem}
\theoremstyle{plain}

\newtheorem{example}{Example}

\newtheorem{lemma}{Lemma}

\input{tcilatex}

\begin{document}
\title[Inverse nodal problem]{Reconstruction of the Dirac-Type
Integro-Differential Operator from Nodal Data}
\author{Baki Keskin}
\author{H. Dilara Tel}
\curraddr{Department of Mathematics, Faculty of Science, Cumhuriyet
University 58140 \\
Sivas, TURKEY}
\email{bkeskin@cumhuriyet.edu.tr}
\curraddr{Department of Mathematics, Faculty of Science, Cumhuriyet
University 58140 \\
Sivas, TURKEY}
\email{dilara\_5820@hotmail.com}

\begin{abstract}
The inverse nodal problem for Dirac type integro-differential operator with
the spectral parameter in the boundary conditions is studied. We prove that
dense subset of the nodal points determines\ the coefficients of
differential part of operator and gives partial information for integral
part of it.
\end{abstract}

\maketitle

\section{\textbf{Introduction}}

We consider the following boundary value problem $L$ generated by the
Dirac-type integro-differential system with the spectral parameter $\lambda $
in the boundary conditions:

\begin{equation}
\ell \left[ Y(x)\right] :=BY^{\prime }(x)+\Omega
(x)Y(x)+\int\limits_{0}^{x}\chi (x,t)Y(t)dt=\lambda Y(x),\text{ \ }x\in
(0,\pi ),
\end{equation}%
subject to the boundary conditions

\begin{eqnarray}
(\lambda \cos \theta +b_{1})y_{1}(0)+(\lambda \sin \theta +b_{2})y_{2}(0)
&=&0\medskip \\
(\lambda \cos \beta +d_{1})y_{1}(\pi )+(\lambda \sin \beta +d_{2})y_{2}(\pi
) &=&0
\end{eqnarray}%
where $\theta $ , $\beta ,$ $b_{1},b_{2},d_{1}$ and $d_{2}$ are real
constants and $\lambda $ is the spectral parameter, $B=\left( 
\begin{array}{cc}
0 & 1 \\ 
-1 & 0%
\end{array}%
\right) ,$ \ $\Omega (x)=\left( 
\begin{array}{cc}
V(x)+m & 0 \\ 
0 & V(x)-m%
\end{array}%
\right) ,$ $\chi (x,t)=\left( 
\begin{array}{cc}
\chi _{11}(x,t) & \chi _{12}(x,t) \\ 
\chi _{21}(x,t) & \chi _{22}(x,t)%
\end{array}%
\right) ,$ \ $Y(x)=\left( 
\begin{array}{c}
y_{1}(x) \\ 
y_{2}(x)%
\end{array}%
\right) ,$ $\Omega (x)$ and $\chi \left( x,t\right) $ are real-valued
functions in the class of $W_{2}^{1}(0,\pi )$, where $m$ is a real constant.
Throughout this paper, we denote $p(x)=V(x)+m,$ $r(x)=V(x)-m$ and assume $%
\int_{0}^{\pi }(p(t)+r(t))dt=0$.

Inverse nodal problems were first proposed and solved for Sturm--Liouville
operator by McLaughlin in 1988 \cite{mc1} In this study, it has been shown
that a dense subset of zeros of eigenfunctions, called nodal points,
uniquelly determines the potential of the Sturm Liouville operator..In 1989,
Hald and McLaughlin gave some numerical schemes for reconstructing potential
from nodal points for more general boundary conditions. \cite{H}. In 1997
Yang gave an algorithm to determine the coefficients of operator for the
inverse nodal Sturm-Liouville problem \cite{yang}. Inverse nodal problems
have been addressed by various researchers in several papers for different
operators (\cite{Br2}, \cite{Buterin}, \cite{buterin2}, \cite{ch}, \cite{Law}%
, \cite{Ozkan}, \cite{Yur}, \cite{Yang3} and \cite{Yang4}). The inverse
nodal problems for Dirac operators with various boundary conditions have
been studied and shown that the dense subsets of nodal points which are the
first components of the eigenfunctions determines the coefficients of
discussed operator in \cite{Guo}, \cite{Yang2} and \cite{Yang5}.

In recent years, perturbation of a differential operator by a Volterra type
integral operator, namely the integro-differential operator have begun to
take significant place in the literature.\cite{But}, \cite{But 2}, \cite{G}, 
\cite{Kur2} and \cite{B}). Integro-differential operators are nonlocal, and
therefore they are more difficult for investigation, than local ones. New
methods for solution of these problems are being developed. For
Sturm-Liouville type integro-differential operators, there exist some
studies about inverse problems but there is very little study for Dirac type
integro-differential operators. The inverse nodal problem for Dirac type
integro-differential operators was first studied by \cite{Kskn}. In their
study, it is shown that the coefficients of the differential part of the
operator can be determined by using nodal points and nodal points also gives
the partial information about integral part. In our study, we deal with an
inverse nodal problem of reconstructing the Dirac type integro-differential
operators with the spectral parameter in the boundary conditions. Firstly,
we have obtained a new approach for asymptotic expressions of the integral
equations of the solutions of such discussed problem. Secondly, more
accurate estimates of eigenvalues and nodal points have been calculated with
the help of these asymptotics. Lastly, We have proved that the operator can
be reconstructed by nodal points.

\section{Main Results}

Let $\varphi (x,\lambda )=\left( 
\begin{array}{c}
\varphi _{1}(x,\lambda ) \\ 
\varphi _{2}(x,\lambda )%
\end{array}%
\right) $ be the solution of (1) under the the initial condition $\varphi
(0,\lambda )=\left( 
\begin{array}{c}
\lambda \sin \theta +b_{2} \\ 
\lambda \cos \theta -b_{1}%
\end{array}%
\right) $. It is easy to see that this solution is an entire function of $%
\lambda $ for each fixed $x$ and $t$. One can easily verify that the
function $\varphi (x,\lambda )$ satisfies%
\begin{equation}
\begin{array}{l}
\varphi _{1}(x,\lambda )=\lambda \sin \left( \theta +\lambda x\right)
+b_{1}\sin \lambda x+b_{2}\cos \lambda x \\ 
+\dint\limits_{0}^{x}p(t)\varphi _{1}(t)\sin \lambda
(x-t)dt+\dint\limits_{0}^{x}r(t)\varphi _{2}(t)\cos \lambda (t-x)dt\medskip
\\ 
+\dint\limits_{0}^{x}\dint\limits_{0}^{t}\left\{ \chi _{11}(t,\eta )\varphi
_{1}(\eta )+\chi _{12}(t,\eta )\varphi _{2}(\eta )\right\} \sin \lambda
(x-t)d\eta dt\medskip \\ 
+\dint\limits_{0}^{x}\dint\limits_{0}^{t}\left\{ \chi _{21}(t,\eta )\varphi
_{1}(\eta )+\chi _{22}(t,\eta )\varphi _{2}(\eta )\right\} \cos \lambda
(x-t)d\eta dt,%
\end{array}%
\end{equation}%
\begin{equation}
\begin{array}{l}
\varphi _{2}(x,\lambda )=-\lambda \cos \left( \theta +\lambda x\right)
-b_{1}\cos \lambda x+b_{2}\sin \lambda x\medskip \\ 
-\dint\limits_{0}^{x}p(t)\varphi _{1}(t)\cos \lambda
(x-t)dt+\dint\limits_{0}^{x}r(t)\varphi _{2}(t)\sin \lambda (t-x)dt\medskip
\\ 
-\dint\limits_{0}^{x}\dint\limits_{0}^{t}\left\{ \chi _{11}(t,\eta )\varphi
_{1}(\eta )+\chi _{12}(t,\eta )\varphi _{2}(\eta )\right\} \cos \lambda
(x-t)d\eta dt\medskip \\ 
+\dint\limits_{0}^{x}\dint\limits_{0}^{t}\left\{ \chi _{21}(t,\eta )\varphi
_{1}(\eta )+\chi _{22}(t,\eta )\varphi _{2}(\eta )\right\} \sin \lambda
(x-t)d\eta dt.%
\end{array}%
\end{equation}

\begin{theorem}
The functions $\varphi _{1}(x,\lambda )$ and $\varphi _{2}(x,\lambda )$ have
the following asymptotic expansions:%
\begin{eqnarray}
\varphi _{1}(x,\lambda ) &=&\lambda \sin \left[ \theta +\lambda x-\nu (x)%
\right] +b_{1}\sin \left[ \lambda x-\nu (x)\right]  \notag \\
&&+b_{2}\cos \left[ \lambda x-\nu (x)\right] +m\cos \theta \sin \left[
\lambda x-\nu (x)\right]  \notag \\
&&+\frac{b_{1}m}{\lambda }\sin \left[ \lambda x-\nu (x)\right] -\frac{m^{2}x%
}{2}\cos \left[ \theta +\lambda x-\nu (x)\right] \medskip  \notag \\
&&-\frac{b_{1}m^{2}x}{2\lambda }\cos \left[ \lambda x-\nu (x)\right] +\frac{%
b_{2}m^{2}x}{2\lambda }\sin \left[ \lambda x-\nu (x)\right] \medskip \\
&&-\frac{K\left( x\right) }{2}\sin \left[ \theta +\lambda x-\nu (x)\right] -%
\frac{b_{1}}{2\lambda }K\left( x\right) \sin \left[ \lambda x-\nu (x)\right]
\medskip  \notag \\
&&-\frac{b_{2}}{2\lambda }K\left( x\right) \cos \left[ \lambda x-\nu (x)%
\right] +\frac{L\left( x\right) }{2}\cos \left[ \theta +\lambda x-\nu (x)%
\right] \medskip  \notag \\
&&+\frac{b_{1}}{2\lambda }L\left( x\right) \cos \left[ \lambda x-\nu (x)%
\right] -\frac{b_{2}}{2\lambda }L\left( x\right) \sin \left[ \lambda x-\nu
(x)\right] +o\left( \frac{e^{\left\vert \tau \right\vert x}}{\lambda }%
\right) ,  \notag
\end{eqnarray}%
\begin{eqnarray}
\varphi _{2}(x,\lambda ) &=&-\lambda \cos \left[ \theta +\lambda x-\nu (x)%
\right] -b_{1}\cos \left[ \lambda x-\nu (x)\right]  \notag \\
&&+b_{2}\sin \left[ \lambda x-\nu (x)\right] -m\sin \theta \sin \left[
\lambda x-\nu (x)\right]  \notag \\
&&-\frac{b_{2}m}{\lambda }\sin \left[ \lambda x-\nu (x)\right] -\frac{m^{2}x%
}{2}\sin \left[ \theta +\lambda x-\nu (x)\right] \medskip  \notag \\
&&-\frac{b_{1}m^{2}x}{2\lambda }\sin \left[ \lambda x-\nu (x)\right] -\frac{%
b_{2}m^{2}x}{2\lambda }\cos \left[ \lambda x-\nu (x)\right] \medskip \\
&&+\frac{K\left( x\right) }{2}\cos \left[ \theta +\lambda x-\nu (x)\right] +%
\frac{b_{1}}{2\lambda }K\left( x\right) \cos \left[ \lambda x-\nu (x)\right]
\medskip  \notag \\
&&-\frac{b_{2}}{2\lambda }K\left( x\right) \sin \left[ \lambda x-\nu (x)%
\right] +\frac{L\left( x\right) }{2}\sin \left[ \theta +\lambda x-\nu (x)%
\right] \medskip  \notag \\
&&+\frac{b_{1}}{2\lambda }L\left( x\right) \sin \left[ \lambda x-\nu (x)%
\right] +\frac{b_{2}}{2\lambda }L\left( x\right) \cos \left[ \lambda x-\nu
(x)\right] +o\left( \frac{e^{\left\vert \tau \right\vert x}}{\lambda }%
\right) .  \notag
\end{eqnarray}%
for sufficiently large $\left\vert \lambda \right\vert ,$ uniformly in $x,$
where, $\nu (x)=\dfrac{1}{2}\int_{0}^{x}(p(t)+r(t))dt=\int_{0}^{x}V(t)dt,$ $%
K(x)=\int_{0}^{x}(\chi _{11}(t,t)+\chi _{22}(t,t))dt,$ $L(x)=\int_{0}^{x}(%
\chi _{12}(t,t)-\chi _{21}(t,t))dt$ and $\tau =\func{Im}\lambda .$
\end{theorem}

\begin{proof}
To apply the method of successive approximations to (4) and (5), put%
\begin{equation*}
\begin{array}{l}
\varphi _{1,0}(x,\lambda )=\lambda \sin \left( \lambda x+\theta \right)
+b_{1}\sin \lambda x+b_{2}\cos \lambda x,\medskip \\ 
\varphi _{1,n+1}(x,\lambda )=\dint\limits_{0}^{x}p(t)\varphi _{1,n}(t)\sin
\lambda (x-t)dt+\dint\limits_{0}^{x}r(t)\varphi _{2,n}(t)\cos \lambda
(t-x)dt\medskip \\ 
+\dint\limits_{0}^{x}\dint\limits_{0}^{t}\left\{ \chi _{11}(t,\eta )\varphi
_{1,n}(\eta )+\chi _{12}(t,\eta )\varphi _{2,n}(\eta )\right\} \sin \lambda
(x-t)d\eta dt\medskip \\ 
+\dint\limits_{0}^{x}\dint\limits_{0}^{t}\left\{ \chi _{21}(t,\eta )\varphi
_{1,n}(\eta )+\chi _{22}(t,\eta )\varphi _{2,n}(\eta )\right\} \cos \lambda
(x-t)d\eta dt,%
\end{array}%
\end{equation*}%
and%
\begin{equation*}
\begin{array}{l}
\varphi _{2,0}(x,\lambda )=-\lambda \cos \left( \theta +\lambda x\right)
-b_{1}\cos \lambda x+b_{2}\sin \lambda x,\medskip \\ 
\varphi _{2,n+1}(x,\lambda )=-\dint\limits_{0}^{x}p(t)\varphi _{1,n}(t)\cos
\lambda (x-t)dt+\dint\limits_{0}^{x}r(t)\varphi _{2,n}(t)\sin \lambda
(t-x)dt\medskip \\ 
-\dint\limits_{0}^{x}\dint\limits_{0}^{t}\left\{ \chi _{11}(t,\eta )\varphi
_{1,n}(\eta )+\chi _{12}(t,\eta )\varphi _{2,n}(\eta )\right\} \cos \lambda
(x-t)d\eta dt\medskip \\ 
+\dint\limits_{0}^{x}\dint\limits_{0}^{t}\left\{ \chi _{21}(t,\eta )\varphi
_{1,n}(\eta )+\chi _{22}(t,\eta )\varphi _{2,n}(\eta )\right\} \sin \lambda
(x-t)d\eta dt.%
\end{array}%
\end{equation*}%
Then we have%
\begin{eqnarray*}
\varphi _{1,1}(x,\lambda ) &=&-\nu (x)\lambda \cos \left( \theta +\lambda
x\right) -\nu (x)b_{1}\cos \lambda x+\nu (x)b_{2}\sin \lambda x\medskip \\
&&+m\cos \theta \sin \lambda x+\frac{b_{1}m}{\lambda }\sin \lambda x-\frac{%
K\left( x\right) }{2}\sin \left( \theta +\lambda x\right) \medskip \\
&&-\frac{b_{1}}{2\lambda }K\left( x\right) \sin \lambda x-\frac{b_{2}}{%
2\lambda }K\left( x\right) \cos \lambda x+\frac{L\left( x\right) }{2}\cos
\left( \theta +\lambda x\right) \medskip \\
&&+\frac{b_{1}}{2\lambda }L\left( x\right) \cos \lambda x-\frac{b_{2}}{%
2\lambda }L\left( x\right) \sin \lambda x+o\left( \frac{e^{\left\vert \tau
\right\vert x}}{\lambda }\right) ,
\end{eqnarray*}%
\begin{eqnarray*}
\varphi _{2,1}(x,\lambda ) &=&-\nu (x)\lambda \sin \left( \theta +\lambda
x\right) -\nu (x)b_{1}\sin \lambda x-\nu (x)b_{2}\cos \lambda x\medskip \\
&&-m\sin \theta \sin \lambda x-\frac{b_{2}m}{\lambda }\sin \lambda x+K\left(
x\right) \cos \left( \theta +\lambda x\right) \medskip \\
&&+\frac{b_{1}}{2\lambda }K\left( x\right) \cos \lambda x-\frac{b_{2}}{%
2\lambda }K\left( x\right) \sin \lambda x+\frac{L\left( x\right) }{2}\sin
\left( \theta +\lambda x\right) \medskip \\
&&+\frac{b_{1}}{2\lambda }L\left( x\right) \sin \lambda x+\frac{b_{2}}{%
2\lambda }L\left( x\right) \cos \lambda x+o\left( \frac{e^{\left\vert \tau
\right\vert x}}{\lambda }\right) ,
\end{eqnarray*}%
and for $n\epsilon 
\mathbb{Z}
^{+}$%
\begin{eqnarray*}
\varphi _{1,2n+1}(x,\lambda ) &=&\left( -1\right) ^{n+1}\frac{\nu
^{2n+1}\left( x\right) }{\left( 2n+1\right) !}\lambda \cos \left( \theta
+\lambda x\right) +\left( -1\right) ^{n+1}b_{1}\frac{\nu ^{2n+1}\left(
x\right) }{\left( 2n+1\right) !}\cos \lambda x\medskip \\
&&+\left( -1\right) ^{n}b_{2}\frac{\nu ^{2n+1}\left( x\right) }{\left(
2n+1\right) !}\sin \lambda x+\left( -1\right) ^{n}m\frac{\nu ^{2n}\left(
x\right) }{\left( 2n\right) !}\cos \theta \sin \lambda x\medskip \\
&&+\left( -1\right) ^{n}\frac{b_{1}m}{\lambda }\frac{\nu ^{2n}\left(
x\right) }{\left( 2n\right) !}\sin \lambda x+\left( -1\right) ^{n}\frac{%
m^{2}x}{2}\frac{\nu ^{2n-1}\left( x\right) }{\left( 2n-1\right) !}\sin
\left( \theta +\lambda x\right) \medskip \\
&&+\left( -1\right) ^{n}\frac{b_{1}m^{2}x}{2\lambda }\frac{\nu ^{2n-1}\left(
x\right) }{\left( 2n-1\right) !}\sin \lambda x+\left( -1\right) ^{n}\frac{%
b_{2}m^{2}x}{2\lambda }\frac{\nu ^{2n-1}\left( x\right) }{\left( 2n-1\right)
!}\cos \lambda x\medskip \\
&&+\left( -1\right) ^{n+1}\frac{\nu ^{2n}\left( x\right) }{\left( 2n\right) !%
}\frac{K\left( x\right) }{2}\sin \left( \theta +\lambda x\right) +\left(
-1\right) ^{n+1}\frac{b_{1}}{2\lambda }\frac{\nu ^{2n}\left( x\right) }{%
\left( 2n\right) !}K\left( x\right) \sin \lambda x\medskip \\
&&+\left( -1\right) ^{n+1}\frac{b_{2}}{2\lambda }\frac{\nu ^{2n}\left(
x\right) }{\left( 2n\right) !}K\left( x\right) \cos \lambda x+\left(
-1\right) ^{n}\frac{\nu ^{2n}\left( x\right) }{\left( 2n\right) !}\frac{%
L\left( x\right) }{2}\cos \left( \theta +\lambda x\right) \medskip \\
&&+\left( -1\right) ^{n}\frac{b_{1}}{2\lambda }\frac{\nu ^{2n}\left(
x\right) }{\left( 2n\right) !}L\left( x\right) \cos \lambda x+\left(
-1\right) ^{n+1}\frac{b_{2}}{2\lambda }\frac{\nu ^{2n}\left( x\right) }{%
\left( 2n\right) !}L\left( x\right) \sin \lambda x\medskip \\
&&+o\left( \frac{e^{\left\vert \tau \right\vert x}}{\lambda }\right) ,
\end{eqnarray*}%
\begin{eqnarray*}
\varphi _{1,2n}(x,\lambda ) &=&\left( -1\right) ^{n}\frac{\nu ^{2n}\left(
x\right) }{\left( 2n\right) !}\sin \lambda x+\left( -1\right) ^{n}\frac{\nu
^{2n}\left( x\right) }{\left( 2n\right) !}\cos \lambda x\medskip \\
&&+\left( -1\right) ^{n}m\frac{\nu ^{2n-1}\left( x\right) }{\left(
2n-1\right) !}\cos \theta \cos \lambda x+\left( -1\right) ^{n}\frac{b_{1}m}{%
\lambda }\frac{\nu ^{2n-1}\left( x\right) }{\left( 2n-1\right) !}\cos
\lambda x\medskip \\
&&+\left( -1\right) ^{n}\frac{m^{2}x}{2}\frac{\nu ^{2n-2}\left( x\right) }{%
\left( 2n-2\right) !}\cos \left( \theta +\lambda x\right) +\left( -1\right)
^{n}\frac{b_{1}m^{2}x}{2\lambda }\frac{\nu ^{2n-2}\left( x\right) }{\left(
2n-2\right) !}\cos \lambda x\medskip \\
&&+\left( -1\right) ^{n+1}\frac{b_{2}m^{2}x}{2\lambda }\frac{\nu
^{2n-2}\left( x\right) }{\left( 2n-2\right) !}\sin \lambda x+\frac{\left(
-1\right) ^{n+1}}{2}\frac{\nu ^{2n-1}\left( x\right) }{\left( 2n-1\right) !}%
K\left( x\right) \cos \left( \theta +\lambda x\right) \medskip \\
&&+\left( -1\right) ^{n+1}\frac{b_{1}}{2\lambda }\frac{\nu ^{2n-1}\left(
x\right) }{\left( 2n-1\right) !}K\left( x\right) \cos \lambda x+\left(
-1\right) ^{n}\frac{b_{2}}{2\lambda }\frac{\nu ^{2n-1}\left( x\right) }{%
\left( 2n-1\right) !}K\left( x\right) \sin \lambda x\medskip \\
&&+\frac{\left( -1\right) ^{n+1}}{2}\frac{\nu ^{2n-1}\left( x\right) }{%
\left( 2n-1\right) !}L\left( x\right) \sin \left( \theta +\lambda x\right)
+\left( -1\right) ^{n+1}\frac{b_{1}}{2\lambda }\frac{\nu ^{2n-1}\left(
x\right) }{\left( 2n-1\right) !}L\left( x\right) \sin \lambda x\medskip \\
&&+\left( -1\right) ^{n+1}\frac{b_{2}}{2\lambda }\frac{\nu ^{2n-1}\left(
x\right) }{\left( 2n-1\right) !}L\left( x\right) \cos \lambda x+o\left( 
\frac{e^{\left\vert \tau \right\vert x}}{\lambda }\right) ,\medskip
\end{eqnarray*}%
\begin{eqnarray*}
\varphi _{2,2n+1}(x,\lambda ) &=&\left( -1\right) ^{n+1}\frac{\nu
^{2n+1}\left( x\right) }{\left( 2n+1\right) !}\sin \lambda x+\left(
-1\right) ^{n+1}\frac{\nu ^{2n+1}\left( x\right) }{\left( 2n+1\right) !}\cos
\lambda x\medskip \\
&&+\left( -1\right) ^{n+1}m\frac{\nu ^{2n}\left( x\right) }{\left( 2n\right)
!}\sin \theta \sin \lambda x+\left( -1\right) ^{n+1}\frac{b_{2}m}{\lambda }%
\frac{\nu ^{2n}\left( x\right) }{\left( 2n\right) !}\sin \lambda x\medskip \\
&&+\left( -1\right) ^{n+1}\frac{m^{2}x}{2}\frac{\nu ^{2n-1}\left( x\right) }{%
\left( 2n-1\right) !}\cos \left( \theta +\lambda x\right) +\left( -1\right)
^{n+1}\frac{b_{1}m^{2}x}{2\lambda }\frac{\nu ^{2n-1}\left( x\right) }{\left(
2n-1\right) !}\cos \lambda x\medskip \\
&&+\left( -1\right) ^{n}\frac{b_{2}m^{2}x}{2\lambda }\frac{\nu ^{2n-1}\left(
x\right) }{\left( 2n-1\right) !}\sin \lambda x+\frac{\left( -1\right) ^{n}}{2%
}\frac{\nu ^{2n}\left( x\right) }{\left( 2n\right) !}K\left( x\right) \cos
\left( \theta +\lambda x\right) \medskip \\
&&+\left( -1\right) ^{n}\frac{b_{1}}{2\lambda }\frac{\nu ^{2n}\left(
x\right) }{\left( 2n\right) !}K\left( x\right) \cos \lambda x+\left(
-1\right) ^{n+1}\frac{b_{2}}{2\lambda }\frac{\nu ^{2n}\left( x\right) }{%
\left( 2n\right) !}K\left( x\right) \sin \lambda x\medskip \\
&&+\frac{\left( -1\right) ^{n}}{2}\frac{\nu ^{2n}\left( x\right) }{\left(
2n\right) !}L\left( x\right) \sin \left( \theta +\lambda x\right) +\left(
-1\right) ^{n}\frac{b_{1}}{2\lambda }\frac{\nu ^{2n}\left( x\right) }{\left(
2n\right) !}L\left( x\right) \sin \lambda x\medskip \\
&&+\left( -1\right) ^{n}\frac{b_{2}}{2\lambda }\frac{\nu ^{2n}\left(
x\right) }{\left( 2n\right) !}L\left( x\right) \cos \lambda x+o\left( \frac{%
e^{\left\vert \tau \right\vert x}}{\lambda }\right) ,
\end{eqnarray*}%
\begin{eqnarray*}
\varphi _{2,2n}(x,\lambda ) &=&\left( -1\right) ^{n+1}\frac{\nu ^{2n}\left(
x\right) }{\left( 2n\right) !}\lambda \cos \left( \theta +\lambda x\right)
+\left( -1\right) ^{n+1}b_{1}\frac{\nu ^{2n}\left( x\right) }{\left(
2n\right) !}\cos \lambda x\medskip \\
&&+\left( -1\right) ^{n}b_{2}\frac{\nu ^{2n}\left( x\right) }{\left(
2n\right) !}\sin \lambda x+\left( -1\right) ^{n+1}m\frac{\nu ^{2n-1}\left(
x\right) }{\left( 2n-1\right) !}\sin \theta \cos \lambda x\medskip \\
&&+\left( -1\right) ^{n+1}\frac{b_{2}m}{\lambda }\frac{\nu ^{2n-1}\left(
x\right) }{\left( 2n-1\right) !}\cos \lambda x+\left( -1\right) ^{n}\frac{%
m^{2}x}{2}\frac{\nu ^{2n-2}\left( x\right) }{\left( 2n-2\right) !}\sin
\left( \theta +\lambda x\right) \medskip \\
&&+\left( -1\right) ^{n}\frac{b_{1}m^{2}x}{2\lambda }\frac{\nu ^{2n-2}\left(
x\right) }{\left( 2n-2\right) !}\sin \lambda x+\left( -1\right) ^{n}\frac{%
b_{2}m^{2}x}{2\lambda }\frac{\nu ^{2n-2}\left( x\right) }{\left( 2n-2\right)
!}\cos \lambda x\medskip \\
&&+\left( -1\right) ^{n+1}\frac{\nu ^{2n-1}\left( x\right) }{\left(
2n-1\right) !}\frac{K\left( x\right) }{2}\sin \left( \theta +\lambda
x\right) +\left( -1\right) ^{n+1}\frac{b_{1}}{2\lambda }\frac{\nu
^{2n-1}\left( x\right) }{\left( 2n-1\right) !}K\left( x\right) \sin \lambda
x\medskip \\
&&+\left( -1\right) ^{n+1}\frac{b_{2}}{2\lambda }\frac{\nu ^{2n-1}\left(
x\right) }{\left( 2n-1\right) !}K\left( x\right) \cos \lambda x+\left(
-1\right) ^{n}\frac{\nu ^{2n-1}\left( x\right) }{\left( 2n-1\right) !}\frac{%
L\left( x\right) }{2}\cos \left( \theta +\lambda x\right) \medskip \\
&&+\left( -1\right) ^{n}\frac{b_{1}}{2\lambda }\frac{\nu ^{2n-1}\left(
x\right) }{\left( 2n-1\right) !}L\left( x\right) \cos \lambda x+\left(
-1\right) ^{n+1}\frac{b_{2}}{2\lambda }\frac{\nu ^{2n-1}\left( x\right) }{%
\left( 2n-1\right) !}L\left( x\right) \sin \lambda x\medskip \\
&&+o\left( \frac{e^{\left\vert \tau \right\vert x}}{\lambda }\right) ,
\end{eqnarray*}%
for sufficiently large $\left\vert \lambda \right\vert ,$ uniformly in $x.$
Hence, the proof \ the theorem1 is completed by successive approximations
method.
\end{proof}

Define the entire function $\Delta (\lambda )$ by 
\begin{equation}
\Delta (\lambda )=\varphi _{1}(\pi ,\lambda )\left( \lambda \cos \beta
+d_{1}\right) +\varphi _{2}(\pi ,\lambda )\left( \lambda \sin \beta
+d_{2}\right) ,
\end{equation}%
this function is called the characteristic function of the problem (1)-(3)
and \ the zeros of it, namely $\left\{ \lambda _{n}\right\} _{n\in 
\mathbb{Z}
},$\ coincide with the eigenvalues of the problem (1)-(3). From (6) and (7),
we get the characteristic function $\Delta (\lambda )$\ has the following
asymptotic relation for sufficiently large $\left\vert \lambda \right\vert ,$%
\begin{eqnarray*}
\Delta (\lambda ) &=&\lambda ^{2}\left\{ \sin \left[ \lambda \pi +\theta
-\beta \right] \right. \medskip \\
&&\left. +\dfrac{b_{1}}{\lambda }\sin \left[ \lambda \pi -\beta \right] +%
\dfrac{b_{2}}{\lambda }\cos \left[ \lambda \pi -\beta \right] \right.
\medskip \\
&&\left. +\dfrac{m}{\lambda }\sin \left[ \lambda \pi \right] \cos (\theta
+\beta )-\dfrac{m^{2}\pi b_{1}}{2\lambda }\cos \left[ \lambda \pi +\theta
-\beta \right] \right. \medskip \\
&&\left. -\dfrac{K(\pi )}{2\lambda }\sin \left[ \lambda \pi +\theta -\beta %
\right] +\dfrac{L(\pi )}{2\lambda }\cos \left[ \lambda \pi +\theta -\beta %
\right] \right. \medskip \\
&&\left. +\dfrac{d_{1}}{\lambda }\sin \left[ \lambda \pi +\theta \right] -%
\dfrac{d_{2}}{\lambda }\cos \left[ \lambda \pi +\theta \right] +o\left( 
\frac{e^{\left\vert \tau \right\vert \pi }}{\lambda }\right) \right\}
.\medskip \\
&&\medskip
\end{eqnarray*}%
Since the the roots of $\Delta (\lambda _{n})=0$ are the eigenvalues of the
problem (1)-(3), the following equation holds\newline
$\tan \left[ \lambda _{n}\pi +\theta -\beta \right] \times \left\{
1+B+o\left( \dfrac{e^{\left\vert \tau \right\vert \pi }}{\lambda _{n}}%
\right) \right\} =C+o\left( \dfrac{e^{\left\vert \tau \right\vert \pi }}{%
\lambda _{n}}\right) \medskip $\newline
where,

$B=\dfrac{b_{1}}{\lambda _{n}}\cos \theta +\dfrac{b_{2}}{\lambda _{n}}\sin
\theta +\dfrac{m}{\lambda _{n}}\cos (\beta -\theta )\cos (\theta +\beta
)\medskip $

$-\dfrac{K(\pi )}{2\lambda _{n}}+\dfrac{d_{1}}{\lambda _{n}}\cos \beta +%
\dfrac{d_{2}}{\lambda _{n}}\sin \beta ,\medskip $

$C=\dfrac{b_{1}}{\lambda _{n}}\sin \theta -\dfrac{b_{2}}{\lambda _{n}}\cos
\theta -\dfrac{m}{\lambda _{n}}\sin (\beta -\theta )\cos (\theta +\beta
)\medskip $

$+\dfrac{m^{2}\pi }{2\lambda _{n}}-\dfrac{L(\pi )}{2\lambda _{n}}-\dfrac{%
d_{1}}{\lambda _{n}}\sin \beta +\dfrac{d_{2}}{\lambda _{n}}\cos \beta
\medskip $\newline
which implies that,$\medskip $

$\tan \left[ \lambda _{n}\pi +\theta -\beta \right] =\left\{ 1-B+o\left( 
\dfrac{e^{\left\vert \tau \right\vert \pi }}{\lambda _{n}}\right) \right\}
\times \left\{ C+o\left( \dfrac{e^{\left\vert \tau \right\vert \pi }}{%
\lambda _{n}}\right) \right\} \medskip $\newline
for sufficiently large $\left\vert n\right\vert .$ It yields from the last
equation that%
\begin{equation}
\lambda _{n}=n+\frac{\beta -\theta +C}{\pi }+o\left( \frac{1}{n}\right)
\medskip
\end{equation}%
for $\left\vert n\right\vert \rightarrow \infty .$

\begin{lemma}
For sufficiently large $\left\vert n\right\vert $, the first component $%
\varphi _{1}(x,\lambda _{n})$ of the eigenfunction $\varphi (x,\lambda _{n})$
has exactly $n-1$ nodes $\left\{ x_{n}^{j}:j=\overline{0,n-2}\right\} $ in
the interval $\left( 0,\pi \right) $ i.e., \newline
$0<x_{n}^{0}<x_{n}^{1}<...<x_{n}^{n-2}<\pi $. The numbers $\left\{
x_{n}^{j}\right\} $ satisfy the following asymptotic formula:%
\begin{eqnarray*}
x_{n}^{j} &=&\frac{j\pi }{n}-\frac{j\pi }{n}\frac{\beta -\theta }{n\pi }+%
\frac{\nu (x_{n}^{j})-\theta }{n}-\left( \nu (x_{n}^{j})-\theta \right)
\left( \frac{\beta -\theta }{n^{2}\pi }\right) \medskip \\
&&+\frac{1}{2n^{2}}\left\{ 2b_{1}\sin \theta -2b_{2}\cos \theta +2m\cos
\theta \sin \theta +m^{2}x_{n}^{j}-L\left( x_{n}^{j}\right) \right\} \medskip
\\
&&+o\left( \frac{1}{n^{2}}\right)
\end{eqnarray*}

\begin{proof}
From (6), we can write%
\begin{eqnarray*}
&&\left. \varphi _{1}(x,\lambda _{n})=\lambda _{n}\sin \left[ \lambda
_{n}x-\nu (x)+\theta \right] +b_{1}\sin \left[ \lambda _{n}x-\nu (x)\right]
\right. \medskip \\
&&\left. +b_{2}\cos \left[ \lambda _{n}x-\nu (x)\right] +m\cos \theta \sin %
\left[ \lambda _{n}x-\nu (x)\right] \right. \medskip \\
&&\left. +\frac{b_{1}m}{\lambda _{n}}\sin \left[ \lambda _{n}x-\nu (x)\right]
-\frac{m^{2}x}{2}\cos \left[ \lambda _{n}x-\nu (x)+\theta \right] \right.
\medskip \\
&&\left. -\frac{b_{1}m^{2}x}{2\lambda _{n}}\cos \left[ \lambda _{n}x-\nu (x)%
\right] +\frac{b_{2}m^{2}x}{2\lambda _{n}}\sin \left[ \lambda _{n}x-\nu (x)%
\right] \right. \medskip \\
&&\left. -\frac{1}{2}\sin \left[ \lambda _{n}x-\nu (x)+\theta \right] K(x)-%
\frac{b_{1}}{2\lambda _{n}}\sin \left[ \lambda _{n}x-\nu (x)\right]
K(x)\right. \medskip \\
&&\left. -\frac{b_{2}}{2\lambda _{n}}\cos \left[ \lambda _{n}x-\nu (x)\right]
K(x)+\frac{1}{2}\cos \left[ \lambda _{n}x-\nu (x)+\theta \right] L(x)\right.
\medskip \\
&&\left. +\frac{b_{1}}{2\lambda _{n}}\cos \left[ \lambda _{n}x-\nu (x)\right]
L(x)-\frac{b_{2}}{2\lambda _{n}}\sin \left[ \lambda _{n}x-\nu (x)\right]
L(x)\right. \medskip \\
&&\left. +o\left( \frac{e^{\left\vert \tau \right\vert x}}{\lambda _{n}}%
\right) \right. \medskip
\end{eqnarray*}%
for sufficiently large $\left\vert n\right\vert .$ From $\varphi
_{1}(x_{n}^{j},\lambda _{n})=0,$ we get%
\begin{eqnarray*}
&&\lambda _{n}\sin \left[ \lambda _{n}x-\nu (x)+\theta \right] +b_{1}\sin %
\left[ \lambda _{n}x-\nu (x)+\theta \right] \cos \theta \medskip \\
&&-b_{1}\cos \left[ \lambda _{n}x-\nu (x)+\theta \right] \sin \theta
+b_{2}\cos \left[ \lambda _{n}x-\nu (x)+\theta \right] \cos \theta \medskip
\\
&&+b_{2}\sin \left[ \lambda _{n}x-\nu (x)+\theta \right] \sin \theta +m\cos
\theta \sin \left[ \lambda _{n}x-\nu (x)+\theta \right] \cos \theta \medskip
\\
&&-m\cos \theta \cos \left[ \lambda _{n}x-\nu (x)+\theta \right] \sin \theta
+\frac{mb_{1}}{\lambda _{n}}\sin \left[ \lambda _{n}x-\nu (x)+\theta \right]
\cos \theta \medskip \\
&&-\frac{mb_{1}}{\lambda _{n}}\cos \left[ \lambda _{n}x-\nu (x)+\theta %
\right] \sin \theta -\frac{m^{2}x}{2}\cos \left[ \lambda _{n}x-\nu
(x)+\theta \right] \medskip \\
&&-\dfrac{b_{1}m^{2}x}{2\lambda _{n}}\cos \left[ \lambda _{n}x-\nu
(x)+\theta \right] \cos \theta -\frac{b_{1}m^{2}x}{2\lambda _{n}}\sin \left[
\lambda _{n}x-\nu (x)+\theta \right] \sin \theta \medskip \\
&&+\frac{b_{2}m^{2}x}{2\lambda _{n}}\sin \left[ \lambda _{n}x-\nu (x)+\theta %
\right] \cos \theta -\frac{b_{2}m^{2}x}{2\lambda _{n}}\cos \left[ \lambda
_{n}x-\nu (x)+\theta \right] \sin \theta \medskip \\
&&-\frac{1}{2}K(x)\sin \left[ \lambda _{n}x-\nu (x)+\theta \right] -\frac{%
b_{1}}{2\lambda _{n}}K(x)\sin \left[ \lambda _{n}x-\nu (x)+\theta \right]
\cos \theta \medskip \\
&&-\frac{b_{1}}{2\lambda _{n}}K(x)\cos \left[ \lambda _{n}x-\nu (x)+\theta %
\right] \sin \theta -\frac{b_{2}}{2\lambda _{n}}K(x)\cos \left[ \lambda
_{n}x-\nu (x)+\theta \right] \cos \theta \medskip \\
&&-\frac{b_{2}}{2\lambda _{n}}K(x)\sin \left[ \lambda _{n}x-\nu (x)+\theta %
\right] \sin \theta +\frac{1}{2}L(x)\cos \left[ \lambda _{n}x-\nu (x)+\theta %
\right] \medskip \\
&&+\frac{b_{1}}{2\lambda _{n}}L(x)\cos \left[ \lambda _{n}x-\nu (x)+\theta %
\right] \cos \theta +\frac{b_{1}}{2\lambda _{n}}L(x)\sin \left[ \lambda
_{n}x-\nu (x)+\theta \right] \sin \theta \medskip \\
&&-\frac{b_{2}}{2\lambda _{n}}L(x)\sin \left[ \lambda _{n}x-\nu (x)+\theta %
\right] \cos \theta +\frac{b_{2}}{2\lambda _{n}}L(x)\cos \left[ \lambda
_{n}x-\nu (x)+\theta \right] \sin \theta \medskip \\
&&\left. +o\left( \frac{e^{\left\vert \tau \right\vert x}}{\lambda _{n}}%
\right) =0\right. \medskip
\end{eqnarray*}%
then we get%
\begin{eqnarray*}
&&\tan \left[ \lambda _{n}x-\nu (x)+\theta \right] +\frac{b_{1}}{\lambda _{n}%
}\tan \left[ \lambda _{n}x-\nu (x)+\theta \right] \cos \theta \medskip \\
&&-\frac{b_{1}}{\lambda _{n}}\sin \theta +\frac{b_{2}}{\lambda _{n}}\cos
\theta +\frac{b_{2}}{\lambda _{n}}\tan \left[ \lambda _{n}x-\nu (x)+\theta %
\right] \sin \theta \medskip \\
&&+\frac{m}{\lambda _{n}}\cos ^{2}\theta \tan \left[ \lambda _{n}x-\nu
(x)+\theta \right] -\frac{m}{\lambda _{n}}\cos \theta \sin \theta \medskip \\
&&-\frac{1}{2\lambda _{n}}K(x)\tan \left[ \lambda _{n}x-\nu (x)+\theta %
\right] -\frac{m^{2}x}{2\lambda _{n}}\medskip \\
&&\left. +\frac{1}{2\lambda _{n}}L(x)+o\left( \frac{1}{\lambda _{n}}\right)
=0\right. \medskip
\end{eqnarray*}%
which is equivalent to%
\begin{eqnarray*}
&&\tan \left[ \lambda _{n}x_{n}^{j}-\nu (x_{n}^{j})+\theta \right] = \\
&&\left\{ 1+\frac{2b_{1}\cos \theta +2b_{2}\sin \theta +2m\cos ^{2}\theta
-K\left( x_{n}^{j}\right) }{2\lambda _{n}}+o\left( \frac{1}{\lambda _{n}}%
\right) \right\} ^{-1}\times \medskip \\
&&\times \left\{ \frac{2b_{1}\sin \theta -2b_{2}\cos \theta +2m\cos \theta
\sin \theta +m^{2}x_{n}^{j}-L\left( x_{n}^{j}\right) }{2\lambda _{n}}%
+o\left( \frac{1}{\lambda _{n}}\right) \right\} \\
\medskip &=&\left\{ 1-\frac{2b_{1}\cos \theta +2b_{2}\sin \theta +2m\cos
^{2}\theta -K\left( x_{n}^{j}\right) }{2\lambda _{n}}+o\left( \frac{1}{%
\lambda _{n}}\right) \right\} \times \\
&&\times \left\{ \frac{2b_{1}\sin \theta -2b_{2}\cos \theta +2m\cos \theta
\sin \theta +m^{2}x_{n}^{j}-L\left( x_{n}^{j}\right) }{2\lambda _{n}}%
+o\left( \frac{1}{\lambda _{n}}\right) \right\} \\
&=&\frac{2b_{1}\sin \theta -2b_{2}\cos \theta +2m\cos \theta \sin \theta
+m^{2}x_{n}^{j}-L\left( x_{n}^{j}\right) }{2\lambda _{n}}+o\left( \frac{1}{%
\lambda _{n}}\right) \medskip
\end{eqnarray*}%
Taylor formula for the function arctangent yields%
\begin{eqnarray*}
&&\lambda _{n}x_{n}^{j}-\nu (x_{n}^{j})+\theta = \\
&&j\pi +\frac{1}{2\lambda _{n}}\left\{ 2b_{1}\sin \theta -2b_{2}\cos \theta
+2m\cos \theta \sin \theta +m^{2}x_{n}^{j}-L\left( x_{n}^{j}\right) \right\}
+o\left( \frac{1}{\lambda _{n}}\right) \medskip
\end{eqnarray*}%
which is equivalent to%
\begin{eqnarray*}
x_{n}^{j} &=&\lambda _{n}^{-1}\left\{ \nu (x_{n}^{j})-\theta +j\pi \right.
\medskip \\
&&\left. +\frac{1}{2\lambda _{n}}\left\{ 2b_{1}\sin \theta -2b_{2}\cos
\theta +2m\cos \theta \sin \theta +m^{2}x_{n}^{j}-L\left( x_{n}^{j}\right)
\right\} \right\} \medskip \\
&&+o\left( \frac{1}{\lambda _{n}^{2}}\right) \medskip
\end{eqnarray*}%
If we put%
\begin{equation*}
\lambda _{n}^{-1}=\frac{1}{n}\left\{ 1-\frac{\nu \left( \pi \right) -\theta
+\beta +C}{n\pi }+o\left( \frac{1}{n}\right) \right\} \medskip
\end{equation*}%
and considiring that $\nu \left( \pi \right) =0$ then we get 
\begin{eqnarray*}
&&x_{n}^{j}=\frac{j\pi }{n}+\frac{\left( \nu (x_{n}^{j})-\theta \right) }{n}-%
\frac{j\pi }{n}\frac{-\theta +\beta }{n\pi }+\left( \nu (x_{n}^{j})-\theta
\right) \left( \frac{-\theta +\beta }{n^{2}\pi }\right) \\
&&+\frac{1}{2n^{2}}\left\{ 2b_{1}\sin \theta -2b_{2}\cos \theta +2m\cos
\theta \sin \theta +m^{2}x_{n}^{j}-L\left( x_{n}^{j}\right) \right\}
+o\left( \frac{1}{n^{2}}\right) \medskip
\end{eqnarray*}
\end{proof}
\end{lemma}

Fix $x\in \left( 0,\pi \right) .$ Let $X$ be the set of nodal points. One
can choose a sequence $\left( x_{n}^{j}\right) \subset X$ such that $%
x_{n}^{j}$ converges to $x.$ Then the following limits are exist and finite:%
\begin{equation*}
f\left( x\right) :=\underset{\left\vert n\right\vert \rightarrow \infty }{%
\lim }\left( x_{n}^{j}-\frac{j\pi }{n}\right) n=-x\left( \frac{\beta -\theta 
}{\pi }\right) +\nu (x)-\theta \medskip
\end{equation*}%
where%
\begin{equation}
f\left( x\right) =-x\left( \frac{\beta -\theta }{\pi }\right) +\frac{1}{2}%
\dint\limits_{0}^{x}\left[ p\left( t\right) +r\left( t\right) \right]
dt-\theta
\end{equation}%
and%
\begin{equation*}
g\left( x\right) :=\underset{\left\vert n\right\vert \rightarrow \infty }{%
\lim }\left( x_{n}^{j}-\frac{j\pi }{n}+\frac{j\pi }{n}\frac{-\theta +\beta }{%
n\pi }-\frac{\nu (x)-\theta }{n}\right) n^{2}\medskip
\end{equation*}%
where

\begin{eqnarray}
g\left( x\right) &=&\left( \nu (x)-\theta \right) \left( \frac{\beta -\theta 
}{\pi }\right) \\
&&+b_{1}\sin \theta -b_{2}\cos \theta +m\cos \theta \sin \theta +\frac{m^{2}x%
}{2}-\frac{L\left( x\right) }{2}  \notag
\end{eqnarray}%
Now,we can formulate the following uniqueness theorem and establish a
constructive procedure for reconstructing the potantial of the considered
problem.

\begin{theorem}
The given dense subset of nodal set $X$ uniquely determines the potential $%
V(x)$ of the problem, the function $L^{\prime }(x)=\chi _{12}(x,x)-\chi
_{21}(x,x)$ of the partial imformation of the integral part, a.e. on $\left(
0,\pi \right) ,$ and the coefficients $\theta $ and $\beta $ of the boundary
conditions. Moreover, $V(x),$ $L^{\prime }(x),$ $\theta $ and $\beta $ can
be constructed by the following algorithm:

(1) fix $x\in (0,\pi ),$ choose a sequence $\left( x_{n}^{j(n)}\right)
\subset X$ such that $\underset{\left\vert n\right\vert \rightarrow \infty }{%
\lim }x_{n}^{j(n)}=x;$

(2)\textbf{\ }find the function $f(x)$ via (10) and calculate 
\begin{eqnarray*}
\theta &=&-f(0) \\
\beta &=&-f(\pi ) \\
V(x) &=&\medskip f^{\prime }(x)+\dfrac{\beta -\theta }{\pi }
\end{eqnarray*}%
(3) find the function $g(x)$ via (11) and calculate%
\begin{eqnarray*}
\medskip m &=&\sqrt{2\frac{g\left( \pi \right) -g\left( 0\right) }{\pi }} \\
L^{^{\prime }}\left( x\right) &=&-2g^{^{\prime }}\left( x\right) -2V(x)\frac{%
\beta -\theta }{\pi }+m^{2}
\end{eqnarray*}

\begin{example}
Let $\left\{ x_{n}^{j(n)}\right\} \subset X$ be the dense subset of nodal
points in $(0,\pi )$ given by the following asimptotics:\newline
\begin{eqnarray*}
x_{n}^{j(n)} &=&\frac{j(n)\pi }{n}+\frac{\left( j(n)\pi /n\right) ^{2}-\pi
j(n)\pi /n-\pi }{4n} \\
&&+\frac{1}{2n^{2}}\left\{ b_{1}\sqrt{2}-b_{2}\sqrt{2}+1+\frac{j(n)\pi }{n}-%
\frac{\pi }{2}\frac{j(n)\pi }{n}+\frac{\left( \frac{j(n)\pi }{n}\right) ^{2}%
}{2}\right\} \\
&&+o\left( \frac{1}{n^{2}}\right)
\end{eqnarray*}%
It can \ be calculated from (10) and (11) that,%
\begin{eqnarray*}
f\left( x\right) &=&\underset{n\rightarrow \infty }{\lim }\left(
x_{n}^{j(n)}-\frac{j(n)\pi }{n}\right) n\medskip \\
&=&\underset{n\rightarrow \infty }{\lim }\left\{ \frac{1}{4}\left( \frac{%
j(n)\pi }{n}\right) ^{2}-\frac{\pi }{4}\frac{j(n)\pi }{n}-\frac{\pi }{4}%
\right. \medskip \\
&&\left. +\frac{1}{2n}\left\{ b_{1}\sqrt{2}-b_{2}\sqrt{2}+1+\frac{j(n)\pi }{n%
}-\frac{\pi }{2}\frac{j(n)\pi }{n}+\frac{1}{2}\left( \frac{j(n)\pi }{n}%
\right) ^{2}\right\} \right. \medskip \\
&&\left. +o\left( \frac{1}{n^{2}}\right) \right\} \medskip \\
&=&\frac{1}{4}x^{2}-\frac{\pi }{4}x-\frac{\pi }{4}\medskip
\end{eqnarray*}%
\begin{eqnarray*}
g\left( x\right) &=&\underset{n\rightarrow \infty }{\lim }\left\{
x_{n}^{j(n)}-\frac{j(n)\pi }{n}-\frac{\frac{1}{4}\left( \dfrac{j(n)\pi }{n}%
\right) ^{2}-\dfrac{j(n)\pi ^{2}}{4n}-\frac{\pi }{4}}{n}\right\}
n^{2}\medskip \\
&=&\underset{n\rightarrow \infty }{\lim }\frac{1}{2}\left\{ b_{1}\sqrt{2}%
-b_{2}\sqrt{2}+1+\frac{j(n)\pi }{n}-\frac{\pi }{4}\frac{j(n)\pi }{n}+\frac{1%
}{4}\left( \frac{j(n)\pi }{n}\right) ^{2}+o\left( \frac{1}{n^{2}}\right)
\right\} \medskip \\
&=&\frac{\sqrt{2}}{2}b_{1}-\frac{\sqrt{2}}{2}b_{2}+\frac{1}{2}+\frac{x}{2}-%
\frac{\pi x}{4}+\frac{x^{2}}{4}\medskip
\end{eqnarray*}%
Therefore, it is obtained \ by using the algorithm in Therem 2,
\end{example}
\end{theorem}

\begin{equation*}
\theta =-f\left( 0\right) =\frac{\pi }{4}\medskip
\end{equation*}%
\begin{equation*}
\beta =-f\left( \pi \right) =\frac{\pi }{4}\medskip
\end{equation*}%
\begin{eqnarray*}
v\left( x\right) &=&f^{\prime }\left( x\right) -\frac{f\left( \pi \right)
-f\left( 0\right) }{\pi } \\
&=&\frac{x}{2}-\frac{\pi }{4}\medskip
\end{eqnarray*}%
\begin{equation*}
m=\sqrt{2\frac{g\left( \pi \right) -g\left( 0\right) }{\pi }}=1\medskip
\end{equation*}%
\begin{equation*}
L^{\prime }(x)=\chi _{12}(x,x)-\chi _{21}(x,x)=-2g^{^{\prime }}\left(
x\right) -2V(x)\frac{\beta -\theta }{\pi }+m^{2}=\frac{\pi }{2}-x\medskip
\end{equation*}

\end{document}